\documentclass[12pt]{amsart}
\usepackage{graphicx}
\usepackage{amsmath}
\usepackage{amsthm}
\usepackage{amsfonts}
\usepackage{amssymb}
\usepackage{color}
\usepackage{verbatim}
\usepackage{cite}
\usepackage{tikz}
\usepackage[margin=1in]{geometry}
\theoremstyle{plain}
\newtheorem{thm}{Theorem}[section]
\newtheorem{lemma}[thm]{Lemma}
\newtheorem{cor}[thm]{Corollary}
\newtheorem{prop}[thm]{Proposition}

\newtheorem*{thm*}{Theorem}

\theoremstyle{remark}
\newtheorem{rmk}[thm]{Remark}
\newtheorem{example}[thm]{Example}

\theoremstyle{definition}
\newtheorem{defn}[thm]{Definition}

\numberwithin{equation}{section}

\def\N{\mathbb{N}}
\def\Z{\mathbb{Z}}

\def\C{\mathcal{C}}
\def\A{\mathcal{A}}
\def\eps{\epsilon}

\DeclareMathOperator{\interior}{Int}
\DeclareMathOperator{\closure}{Cl}

\newenvironment{Pfof1-1-1}
{\par\vskip2\parsep\noindent{\sc Proof of Theorem \ref{Thm:Transitivity} Part 1. }}{{\hfill
		$\Box$}
	\par\vskip2\parsep}
\newenvironment{Pfof1-1-2}
{\par\vskip2\parsep\noindent{\sc Proof of  Theorem \ref{Thm:Transitivity} Part 2. }}{{\hfill
		$\Box$}
	\par\vskip2\parsep}
\newenvironment{Pfof1-2-1}
{\par\vskip2\parsep\noindent{\sc Proof of  Theorem \ref{Thm:Mixing} Part 1. }}{{\hfill
		$\Box$}
	\par\vskip2\parsep}
\newenvironment{Pfof1-2-2}
{\par\vskip2\parsep\noindent{\sc Proof of  Theorem \ref{Thm:Mixing} Part 2. }}{{\hfill
		$\Box$}
	\par\vskip2\parsep}
\newenvironment{Pfof1-3}
{\par\vskip2\parsep\noindent{\sc Proof of Theorem\ \ref{Thm:Specification}. }}{{\hfill
		$\Box$}
	\par\vskip2\parsep}

\begin{document}
	
	\title[Topological dynamics of Markov multi-maps]{Topological dynamics of Markov multi-maps of the interval}
	\author[J. P. Kelly]{James P. Kelly}
	\address[J. P. Kelly]{Department of Mathematics, Christopher Newport University, Newport News, VA 23606, USA}
	\email{james.kelly@cnu.edu}
	\author[K. McGoff]{Kevin McGoff}
	\address[K. McGoff]{Department of Mathematics and Statistics, University of North Carolina at Charlotte, Charlotte, NC 28223, USA}
	\email{kmcgoff1@uncc.edu}
	
	\makeatletter
	\@namedef{subjclassname@2020}{%
		\textup{2020} Mathematics Subject Classification}
	\makeatother
	
	
	\begin{abstract}
	We study Markov multi-maps of the interval from the point of view of topological dynamics. Specifically, we investigate whether they have various properties, including topological transitivity, topological mixing, dense periodic points, and specification. To each Markov multi-map, we associate a shift of finite type (SFT), and then our main results relate the properties of the SFT with those of the Markov multi-map. These results complement existing work showing a relationship between the topological entropy of a Markov multi-map and its associated SFT. We also characterize when the inverse limit systems associated to the Markov multi-maps have the properties mentioned above.
	\end{abstract}
	
	\maketitle

	\section{Introduction} \label{Sect:Intro}
	
	The topological dynamics of multi-maps (also called set-valued functions) have been studied in various contexts with some early examples including \cite{Akin1993,MillerAkin1999}. In the past decade, there has been a renewed focus on the topic \cite{NallKennedy2018,KellyTennant2017,Kawamura2020,CordeiroPacifico2016,MetzgerMoralesRojasThieullen2017,GoodGreenwoodUresin2020}.

	A multi-map of the interval is a function $F : [0,1] \to 2^{[0,1]}$, where $2^{[0,1]}$ denotes the set of closed subsets of $[0,1]$. To any such multi-map one may associate two natural dynamical systems: the forward trajectory system and the inverse limit system. The state space of the forward trajectory system is $X = X(F) \subset [0,1]^{\mathbb{Z}_{\geq 0}}$, where
	\begin{equation*}
		X = \bigl\{ x \in [0,1]^{\mathbb{Z}_{\geq 0}} \colon \forall k \geq 0, \, x_{k+1} \in F(x_k) \bigr\}.
	\end{equation*}
	The natural map for this system is given by the left-shift $T : X \to X$, defined by setting $T(x)_k = x_{k+1}$. In this way, the multi-map $F$ gives rise to the topological dynamical system $(X,T)$. 
	
	To construct the inverse limit system, we begin with the state space $Y = Y(F) \subset [0,1]^{\mathbb{Z}_{\leq 0}}$, where
	\begin{equation*}
		Y = \bigl\{ y \in [0,1]^{\mathbb{Z}_{\leq 0}} \colon \forall k \leq -1, \, y_{k+1} \in F(y_k) \bigr\}.
	\end{equation*}
	The natural map for this system is given by the right-shift, $S \colon Y \to Y$, defined by setting $S(y)_k = y_{k-1}$. Observe that while $F$ is a multi-map, both $T$ and $S$ are continuous, single-valued functions. This enables us to discuss the dynamics of $F$ in terms of the systems $(X,T)$ and $(Y,S)$.
	
	For general multi-maps of the interval, the forward trajectory system and the inverse limit system can exhibit a wide variety of behaviors. Over the last decade, there has been extensive research, in particular, on the topology of the inverse limit space with some recent examples including \cite{GreenwoodKennedy2017,Kato2017,BanicGreenwood2020,NallVidal-Escobar2021,Ingram2021}. In the present work, we focus on the case of Markov multi-maps of the interval. In this context, one may associate to any Markov multi-map of the interval a shift of finite type (SFT) that captures the combinatorics of the multi-map. Markov multi-maps have been studied in recent years with a focus on how the associated SFT can be used to investigate the topological structure of the inverse limit \cite{BanicCrepnjak2018,BanicLunder2016,CrepnjakLunder2016,AlvinKelly2018}, as well as its topological entropy \cite{AlvinKelly2019,KellyMcGoff2019}. 	
	
	In this work, we establish tight connections between the topological dynamics of the forward trajectory system and the associated SFT on the one hand (see Theorems \ref{Thm:Transitivity}, \ref{Thm:Mixing}, and \ref{Thm:Specification}), and between the forward trajectory system and the inverse limit system on the other hand (see Propositions ). Taken together, these results allow one to understand the topological dynamics of both the forward trajectory and inverse limit systems for Markov multi-maps in terms of their associated SFTs. We specifically focus on topological transitivity, density of periodic points, topological mixing, and the specification property.

	In Section~\ref{Sect:Baground}, we give background definitions and notation, and, in Section~\ref{Sect:MarkovDefinitions}, we give a definition for Markov multi-maps and establish additional terminology and notation related to them. We establish some preliminary results in Section~\ref{Sect:PreliminaryResults} before stating and proving our main results in Sections~\ref{Sect:TransandMixing} and \ref{Sect:Spec}. In Section~\ref{Sect:Additional}, we show a connection between the dynamics of the forward and inverse trajectory spaces and other results demonstrating the utility of the main results. Finally, we present two examples, one to which our results apply and one to which they do not, in Section~\ref{Sect:Examples}.

	\section{Background and notation}\label{Sect:Baground}
	
	A topological dynamical system consists of a compact metric space $(X,d)$ and a continuous function $f\colon X\to X$.  We define $f^1=f$, and given $n\geq 2$, we define $f^n=f\circ f^{n-1}$. A point $x\in X$ is called \emph{periodic} under $f$ if $f^p(x)=x$ for some $x\in X$. The system $(X,f)$ is \emph{topologically transitive} if for every pair of non-empty, open sets $U,V\subset X$, we have $f^n(U)\cap V\neq\emptyset$ for some $n\geq 0$. We say that $(X,f)$ is \emph{Devaney chaotic} if it is topologically transitive and the set of periodic points is dense in $X$. This is not  the original definition given by Devaney in \cite{Devaney1989}, but it was shown to be equivalent in \cite{BanksEtAl1992}.
	
	The system $(X,f)$ is \emph{topologically mixing} if for every pair of non-empty, open sets $U,V\subset X$, there exists $N\geq0$ such that $f^n(U)\cap V\neq\emptyset$ for all $n\geq N$. Clearly topological mixing implies topological transitivity. In some contexts, it also implies Devaney chaos, but this is not always the case. For a thorough treatment of these properties and their relationships to one another, see \cite{BrinStuck2015}.
	
	The specification property was introduced in \cite{Bowen1971} as a sufficient condition for positive topological entropy,  but has since been studied extensively as a property unto its own. (See \cite{KwietniakLackaOprocha2016} for a thorough discussion of the specification property and its variants.) The system $(X,f)$ has the \emph{specification property} provided that for all $\eps>0$, there exists $N\geq 0$ such that given any $x_1,\ldots,x_l\in X$ and positive integers
	\[
	a_1\leq b_1<a_2\leq b_2<\cdots<a_l\leq b_l
	\]
	with $a_{i+1}-b_i\geq N$, for all $i=0,\ldots,l-1$, and given any $p\geq b_l+N$, there exists a point $z\in X$ such that $f^p(z)=z$ and, for all $k=0,\ldots,l$ and all $a_k\leq i\leq b_k$, $d(f^i(z),f^i(x_k))<\eps$. Every dynamical system with the specification property also has topological mixing and Devaney chaos.
	
	\subsection{Multi-maps}
	
	We denote by $2^{[0,1]}$ the set of closed subsets of $[0,1]$, and a multi-map of the interval is a function $F\colon[0,1]\to2^{[0,1]}$. The \emph{graph} of $F$ is the set $G(F)=\{(x,y)\in[0,1]^2\colon y\in F(x)\}$. A \emph{forward trajectory} for $F$ is a sequence $(x_0,x_1,\ldots)\in[0,1]^{\Z_{\geq0}}$ such that $(x_i,x_{i+1})\in G(F)$ for all $i\geq 0$, and an \emph{inverse trajectory} for $F$ is a sequence $(x_0,x_{-1},\ldots)\in[0,1]^{Z_{\leq0}}$ such that $(x_{i-1},x_i)\in G(F)$ for all $i\leq 0$. As mentioned in Section \ref{Sect:Intro} above, we denote the sets of forward and inverse trajectories by $X$ and $Y$, respectively, and we consider the appropriate shift maps $T: X \to X$ and $S : Y \to Y$. We endow $X$ and $Y$ with the topologies they inherit as subspaces of $[0,1]^{\Z_{\geq 0}}$ and $[0,1]^{\Z_{\leq0}}$ respectively and note that they are metrizable spaces. 
	
	We  also wish to consider finite versions of these spaces. In the finite case, there is no need for distinction between a forward or inverse trajectory. A \emph{finite trajectory} (of length $n \in \mathbb{N}$) for $F$ is a finite sequence $(x_0,\ldots,x_{n-1})\in[0,1]^n$ where $(x_i,x_{i+1})\in G(F)$ for all $i=0,\ldots,n-2$. We denote the set of finite trajectories of length $n$ by $X_n$ and give this set the subspace topology inherited from  $[0,1]^n$. 
	
	\subsection{Shifts of finite type}
	
	Let $\mathcal{A}$ be a finite set. The one-sided full shift on $\mathcal{A}$ is the set $\Sigma = \mathcal{A}^{\Z_{\geq 0}}$, which we endow with the product topology from the discrete topology on $\mathcal{A}$. Note that this makes $\Sigma$ into a compact metrizable space. There is a natural transformation $\sigma : \Sigma \to \Sigma$ defined for $\mathbf{a} = (a_0,a_1,a_2,\dots) \in \Sigma$ by $\sigma(\mathbf{a})_n = a_{n+1}$ and called the left-shift map on $\Sigma$. The pair $(\Sigma,\sigma)$ and its subsystems are very well-studied topological dynamical systems (see \cite{Lind1995} for a thorough introduction).	The set $\mathcal{A}$ is called the \emph{alphabet} for $\Sigma$.
	
	A \emph{subshift} of $\Sigma$ is any closed set $Z\subset\Sigma$ with $\sigma(Z)\subset Z$. We will be focused on a particular type of subshift defined by a \emph{transition matrix} $M : \mathcal{A} \times \mathcal{A} \to \{0,1\}$. Given a finite alphabet $\A$ and a transition matrix $M$ on $\A$, we define the \emph{nearest neighbor shift of finite type (SFT)} by
	\begin{equation*}
		\Sigma_M = \bigl\{ \mathbf{a} \in \Sigma : \forall n \geq 0, \, M(a_n,a_{n+1}) =1 \bigl\}.
	\end{equation*}
	Note that $\Sigma_M$ is a compact subset of $\Sigma$, and $\Sigma_M$ is invariant under $\sigma$, in the sense that $\sigma(\Sigma_M) \subset \Sigma_M$.
	
	A finite sequence of elements of $\A$ is referred to as a \emph{word}. We denote the set of words of length $n$ appearing in elements of $\Sigma_M$ by $\mathcal{L}_n=\mathcal{L}_n(\Sigma_M)$, and we let $\mathcal{L}=\bigcup \mathcal{L}_n$. The set $\mathcal{L}$ is called the \emph{language} associated with $\Sigma_M$. When necessary, we may look at subsets of the language. If $Z\subset\Sigma_M$ is any subset, then we use $\mathcal{L}_n(Z)$ and $\mathcal{L}(Z)$ to refer to the corresponding sets words appearing in any point of $Z$.
	
	Given a subset $\C\subset\A$, we let $\Sigma_M(\C)$ represent the set of all elements in $\Sigma_M$ which use only symbols from $\C$. For ease of notation, we let $\mathcal{L}_n(\C)=\mathcal{L}(\Sigma_M(\C))$ and $\mathcal{L}_n(\C)=\mathcal{L}_n(\Sigma_M(\C))$.
	
	Let $\Sigma_M$ be a nearest neighbor SFT. A set $\C \subset\A$ is \emph{irreducible} if for every $a,b\in\C$, there is a word $u\in\mathcal{L}(\C)$ beginning with $a$ and ending with $b$. We call $\C$ an \emph{irreducible component} of $\A$ if it is a maximal irreducible subset of $\A$. A set $\C\subset\A$ is \emph{mixing} if there exists $N\geq 1$ such that for every $a,b\in\C$ and every $n\geq N$, there is a word $u\in\mathcal{L}_n(\C)$ beginning with $a$ and ending with $b$. A maximal mixing subset of $\A$ is called a \emph{mixing component}. Given a set $\C\subset\A$, the subsystem $\Sigma_M(\C)$ is topologically transitive (resp. topologically mixing) if and only if the set $\C$ is irreducible (resp. mixing).
	
	\section{Markov multi-maps and their associated SFTs}\label{Sect:MarkovDefinitions}
	
	In this section we introduce Markov multi-maps. We give the definition and discuss some of their properties. We also define the SFT associated to a Markov multi-map, and we present the properties of these associated SFTs that are referenced in our main results (Theorems~\ref{Thm:Transitivity}, \ref{Thm:Mixing}, and \ref{Thm:Specification}).
	
	
	\begin{defn}\label{Defn:MarkovMultiMap}
		A Markov multi-map of $[0,1]$ is defined by a tuple 
		\[
		F = (P, \mathcal{A}_0, \mathcal{A}_1,\mathcal{A}_2, D, R, \{ f_a \}_{a \in \mathcal{A}_0})
		\]
		satisfying the following conditions:
		\begin{enumerate}
			\item $P = \{p_0,\dots,p_r\}$ is a finite partition of the interval such that $0 = p_0 < \dots < p_r = 1$;
			\item $\mathcal{A} = \mathcal{A}_0 \sqcup \mathcal{A}_1 \sqcup \mathcal{A}_2$ is a finite set;
			\item $D \colon \mathcal{A} \to 2^{[0,1]}$, and for each $a \in \mathcal{A}$, there exists $p_i \in P$ such that 
			\begin{equation*}
				D(a) = \left\{ \begin{array}{ll}
					[p_i,p_{i+1}], & \text{if }  a \in \mathcal{A}_0 \\
					\{p_i\}, & \text{if } a \in \mathcal{A}_1 \cup \mathcal{A}_2;
				\end{array}
				\right.
			\end{equation*}
			\item $R \colon \mathcal{A} \to 2^{[0,1]}$, and for each $a \in \mathcal{A}$, there exists $u \leq v$ in $P$ such that $R(a) = [u,v]$, where we also require that 
			\begin{equation*}
				\left\{ \begin{array}{ll}
					u < v, & \text{if } a \in \mathcal{A}_0 \\
					u < v \text{ and } R(a) \cap P = \{u,v\}, & \text{if } a \in \mathcal{A}_1 \\
					u = v, & \text{if } a \in \mathcal{A}_2;
				\end{array}
				\right.
			\end{equation*}
			\item for each $a \in \mathcal{A}_0$, $f_a : D(a) \to R(a)$ is a homeomorphism;
			\item $[0,1] = \bigcup_{a \in \mathcal{A}} D(a)$.
		\end{enumerate}
	\end{defn}
	
	Such a tuple defines a multi-map $F \colon [0,1] \to 2^{[0,1]}$ by the following rule: for $x \in [0,1]$, we let 
	\begin{equation*}
		F(x) = \left( \bigcup_{ \substack{a \in \mathcal{A}_0  \\ x \in D(a)}} \{f_a(x)\} \right) \cup \left( \bigcup_{ \substack{a \in \mathcal{A}_1 \cup \mathcal{A}_2  \\ x \in D(a)}} R(a) \right).
	\end{equation*}
	
	For $a \in \mathcal{A}_0$, let $G(a)$ denote the graph of $f_a$. 
	For $a \in \mathcal{A}_1 \cup \mathcal{A}_2$, let $G(a) = D(a) \times R(a)$. 
	Then the graph of $F$ is
	\begin{equation*}
		G(F) = \bigcup_{a \in \mathcal{A}} G(a).
	\end{equation*}
	Note that each $G(a)$ is closed in $[0,1] \times [0,1]$, and so is $G(F)$.

	Now we make some additional graph-related definitions. 
	\begin{defn}
		Let $a \in \mathcal{A}$.
		\begin{itemize}
			\item Suppose $a \in \mathcal{A}_0$ with $D(a) = [p_i,p_{i+1}]$ and $R(a) = [u,v]$. Define $D_0(a) = (p_i,p_{i+1})$ and $R_0(a) = (u,v)$, and let $G_0(a)$ be the graph of $f_a|_{D_0(a)}$.
			
			\item Suppose $a \in \mathcal{A}_1$ with $D(a) = \{p\}$ and $R(a) = [u,v]$. Define $D_0(a) = \{p\}$ and $R_0(a) = (u,v)$, and let $G_0(a) = \{p\} \times R_0(a)$.
			
			\item Suppose $a \in \mathcal{A}_2$, with $D(a) = \{p\}$ and $R(a) = \{q\}$. Define $D_0(a) = \{p\}$ and $R_0(a) = \{q\}$, and let $G_0(a) = \{(p,q)\}$.
		\end{itemize}
	\end{defn}

	With these definitions, we may now say what it means for a Markov multi-map to be properly parametrized.
	\begin{defn}\label{Defn: Properly parametrized}
		A Markov multi-map is \textit{properly parametrized} if the collection $\{G_0(a) \colon a \in \mathcal{A} \}$ forms a partition of $G(F)$.
	\end{defn}
	Throughout the paper we assume that all Markov multi-maps are properly parametrized. (As was shown in \cite{KellyMcGoff2019}, if $F_0$ is any Markov multi-map,  there exists a properly parametrized Markov multi-map $F_1$ such that $G(F_0) = G(F_1)$.) We also note that Section \ref{Sect:Examples} contains two examples of properly parametrized Markov multi-maps.
	
	\subsection{Associated SFT}

	Let $F$ be any properly parametrized Markov multi-map. Observe that for any $a,b \in \mathcal{A}$, if $D_0(b) \cap R_0(a) \neq \varnothing$, then $D_0(b) \subset R_0(a)$. This property (sometimes called the Markov property) enables us to make connections between $F$ and the following SFT, which we associate to $F$. 
	
	\begin{defn}
		Let $M$ be the square matrix indexed by $\mathcal{A}$ such that for $a,b \in \mathcal{A}$,
		\begin{equation*}
			M(a,b) = \left\{ \begin{array}{ll} 
				1, & \text{if } D_0(b) \subset R_0(a) \\
				0, & \text{otherwise}.
			\end{array}
			\right.
		\end{equation*}
		Let $\Sigma_M \subset \mathcal{A}^{\mathbb{Z}_{\geq0}}$ be the nearest-neighbor SFT with alphabet $\mathcal{A}$ and transition matrix $M$. 
	\end{defn}
	
	In what follows we will use the language of the SFT $\mathcal{L}=\mathcal{L}(\Sigma_M)$ to code finite and infinite trajectories. We make this notion explicit in the following definition.
	
	\begin{defn} Let $n \geq 2$. 
		We say that $(x,u) \in X_n \times \mathcal{L}_{n-1}$ is a \textit{labeled finite trajectory} (of length $n$) if 
		\[
		(x_k,x_{k+1}) \in G(u_k), \quad \forall k = 0,\dots,n-2.
		\]
		We denote the set of  labeled finite trajectories of length $n$ by $\mathcal{T}_{n}$ and endow it with the subspace topology inherited from $[0,1]^{n} \times \mathcal{A}^{n-1}$ (which has the product of the usual topology on $[0,1]^{n}$ and the discrete topology on $\mathcal{A}^{n-1}$).
		
		We say that $(x,u) \in \mathcal{T}_n$ is a \textit{special labeled finite trajectory} (of length $n$) if
		\[
		(x_k,x_{k+1}) \in G_0(u_k), \quad \forall k = 0,\dots,n-2.
		\]
		We denote the set of special labeled finite trajectories of length $n$ by $\mathcal{S}_{n}$ and endow it with the subspace topology inherited from $\mathcal{T}_n$.
	\end{defn}
	
	\subsection{Family of inverse functions}
	
	From Definition~\ref{Defn:MarkovMultiMap}, if $a\in\A_0$, then $f_a$ is a homeomorphism, and thus we may define $g_a=f^{-1}_a$. For $a\in\A_1\cup\A_2$, recall that $D(a)$ is a singleton, and then we define $g_a\colon R(a)\to D(a)$ to be the constant function. In this way, for every $a\in\A$, we have a continuous function $g_a\colon R(a)\to D(a)$ such that $G(a)=\{(x,y)\colon y\in R(a), x=g_a(y)\}$.
	
	We now define inverse functions for all words in $\mathcal{L}$. Given $u=u_0\cdots u_n\in\mathcal{L}$, the definition of $M$ guarantees that $D(u_{i})\subset R(u_{i-1})$ for each $i=1,\ldots,n$, and thus we may define
	\[
	g_u=g_{u_0}\circ\cdots\circ g_{u_n}.
	\]
	
	This brings us to an additional piece of notation to define. For every $u=u_0\cdots u_n\in\mathcal{L}$ we define $I_u=g_u(R(u_n))$. Observe that
	\begin{equation}
		I_{u_0\cdots u_n}\subset I_{u_0\cdots u_{n-1}}\subset\cdots\subset I_{u_0}= D(u_0).\label{Iu}
	\end{equation}
	In many cases, the subset relationships in \eqref{Iu} are not proper. In fact, there is one specific condition that makes one of these proper.	Recall, again from Definition~\ref{Defn:MarkovMultiMap}, that if $a\in\A_0$, then $R(a)$ is an interval with endpoints from the partition $P$, but the endpoints do not need to be adjacent within the partition. If $ab\in\mathcal{L}$, then $I_{ab}$ is a proper subset of $I_a$ if and only if $a\in\A_0$ and the endpoints of $R(a)$ are not adjacent partition elements.
	
	\subsection{Conditions on the associated SFT}
	
	The main theorems of this paper make reference to three conditions on $\Sigma_M$, which  we define here. The first arises from concepts discussed in the previous subsection. Given a (possibly degenerate) interval $I=[p,q]$, we define $\ell(I)=q-p$. Also, for a subshift $Z$ on alphabet $\mathcal{A}$, we say that $Z$ is \textit{transitive on symbols} if $\forall a,b \in \mathcal{L}_1(Z)$, there exists a word in $\mathcal{L}(Z)$ that starts with $a$ and ends with $b$.
	
	\begin{defn}\label{Def:CC}
		We say that $\Sigma_M$ satisfies the \emph{Coding Condition} (CC) if there exists a  subshift $Z\subset\Sigma_M$ such that $Z$ is transitive on symbols and
		\begin{enumerate}
			\item\label{CC1} For all $y\in[0,1]$, there exists $\mathbf{a}\in Z$ and $x\in X$ such that $(x_i,x_{i+1})\in G(a_i)$ for all $i\geq 0$ and $x_0=y$;
			\item $\lim\limits_{n\to\infty}\max\limits_{u\in\mathcal{L}_n(Z)}\ell(I_u)=0$.
		\end{enumerate}
	\end{defn}

	Property (1) in this definition guarantees that the subshift $Z$ codes trajectories that start at any point $y \in [0,1]$. Property (2) can be viewed as a generalization of the notion of uniformly expanding for (single-valued) Markov maps of the interval. In Section~\ref{Sect:Additional}, we give sufficient conditions for the Coding Condition that are relatively easy to verify.
	
	Before proceeding to the next definition, notice that there may be redundancy in how some trajectories in $X$ are coded in $\Sigma_M$. Recall that if $a\in\A_0\cup\A_1$, then $G(a)$ is an arc in $[0,1]^2$. Each endpoint of that arc must be separately represented by an element of $\A_2$. Thus there are finite trajectories that may be coded by a word in $\mathcal{L}(\A_2)$ and also by a word in $\mathcal{L}(\A_0\cup\A_1)$. The following definition serves to distinguish those elements of $\A$ that are ``essential" in coding the trajectories in $X$.
	
	\begin{defn}
		We say that a word $u\in\mathcal{L}_n$ is called \emph{essential} if the set
		$
		\{x\in X_{n+1}\colon (x,u)\in\mathcal{S}_{n+1}\}
		$
		is open in $X_{n+1}$. An element $a\in\mathcal{A}$ is called \emph{essential} if it appears in an essential word. We denote the set of essential elements of $\A$ by $\A^{(e)}$, and we define the set of essential symbolic sequences to be
		\[
		E=\{\mathbf{a}\in\Sigma_M\colon\forall i\geq 0,  a_i\in\mathcal{A}^{(e)}\}.
		\]
	\end{defn}
	
	In Lemmas~\ref{Lemma:EssentialCovers} and \ref{Lemma:essential implies unique labeling}, we show that every trajectory can be coded using exactly one sequence of essential elements of $\A$.
	
	\begin{defn}
		We say that $\Sigma_M$ satisfies the \emph{Irreducibility Condition} (IC) if there is an irreducible component $\C\subset\A$ such that $\A^{(e)} \subset \C$. Similarly, we say that $\Sigma_M$ satisfies the \emph{Mixing Condition} (MC) if there is a mixing component $\C\subset\A$ such that $\A^{(e)} \C$.
	\end{defn}
	
	We have now defined all the terminology used in our main results (Theorems~\ref{Thm:Transitivity}, \ref{Thm:Mixing}, and \ref{Thm:Specification}).

	\section{Preliminary results}\label{Sect:PreliminaryResults}
	
	In this section we establish  results that will be useful in the proofs of our main theorems. We assume throughout that $F$ is a properly parametrized Markov multi-map, $X$ is its forward trajectory space, and $\Sigma_M$ is the associated SFT. We begin with some comments about the topological structure of the of the space of finite trajectories. 
	
	\begin{rmk}\label{Remark:Homeomorphic}
		For every word $u=u_0\cdots u_{n-1}\in\mathcal{L}_n$, $R(u_{n-1})$ is homeomorphic to the set
		$
		\left\{x\in X_{n+1}\colon (x,u)\in\mathcal{T}_{n+1}\right\}.
		$
		Specifically, using the family of inverse functions $\{g_a\colon a\in\A\}$, we may define a homeomorphism 
		\[
		\phi\colon R(u_{n-1})\to \left\{x\in X_{n+1}\colon (x,u)\in\mathcal{T}_{n+1}\right\}
		\]
		by $\phi(x)=(y_i)_{i=0}^{n+1}$ where $y_i=g_{u_i}\circ\cdots\circ g_{u_{n-1}}(x)$ for $0\leq i\leq n$, and $y_{n+1}=x$. Likewise, $R_0(u_{n-1})$ is homeomorphic to the set
		$
		\left\{x\in X_{n+1}\colon (x,u)\in\mathcal{S}_{n+1}\right\}.
		$
	\end{rmk}
	
	From this, we get that for $u\in\mathcal{L}_n$ with $u_{n-1}\in\A_2$, then
	\[
	\left\{x\in X_{n+1}\colon (x,u)\in\mathcal{S}_{n+1}\right\}=\left\{x\in X_{n+1}\colon (x,u)\in\mathcal{T}_{n+1}\right\},
	\]
	which is a singleton set. On the other hand, if $u_{n-1}\in\A_0\cup\A_1$, then $\left\{x\in X_{n+1}\colon (x,u)\in\mathcal{T}_{n+1}\right\}$ is an arc, and $\left\{x\in X_{n+1}\colon (x,u)\in\mathcal{S}_{n+1}\right\}$ is the same arc but with the endpoints removed. This leads us to the following remark.
	
	\begin{rmk}\label{Remark:Closure}
		For all $u\in\mathcal{L}_n$, we have
		\[
			\closure_{X_{n+1}}\left(\left\{x\in X_{n+1}\colon (x,u)\in\mathcal{S}_{n+1}\right\}\right)=\left\{x\in X_{n+1}\colon (x,u)\in\mathcal{T}_{n+1}\right\}.
		\]
	\end{rmk}
	
	This understanding of the topological structure allows us to make the following assertions about essential words.
	
	\begin{lemma}\label{Lemma:EssentialCharacterization}
		Let $u=u_0\cdots u_{n-1}\in\mathcal{L}_n$.
		\begin{enumerate}
			\item\label{EssentialCharacterization1} If $u_{n-1}\in\A_0\cup\A_1$, then $u$ is essential. In particular, $\A_0\cup\A_1\subseteq\A^{(e)}$.
			\item\label{EssentialCharacterization2} If $u_{n-1}\in\A_2$ and  $x\in X_{n+1}$ is the unique trajectory with $(x,u)\in\mathcal{T}_{n+1}$, then $u$ is essential if and only if $x$ is isolated in $X_{n+1}$.
			\end{enumerate}
	\end{lemma}
	\begin{proof}
		We have that
		\[
		X_{n+1}=\bigcup_{u\in\mathcal{L}_n}\{x\in X_{n+1}\colon (x,u)\in\mathcal{T}_{n+1}\}.
		\]
		Since $\mathcal{L}_n$ is finite, $X_{n+1}$ is a finite union of arcs and individual points (all or some of which are contained in those arcs). Moreover, since $F$ is properly parametrized, two arcs cannot intersect except at their endpoints. It follows that if $u_{n-1}\in\A_0\cup\A_1$, then
		\[
		\interior_{X_{n+1}}\left(\left\{x\in X_{n+1}\colon (x,u)\in\mathcal{T}_{n+1}\right\}\right)=\left\{x\in X_{n+1}\colon (x,u)\in\mathcal{S}_{n+1}\right\},
		\]
		so $u$ is essential, thus showing \eqref{EssentialCharacterization1}.
		
		Now suppose $u_{n-1}\in\A_2$, and let $x\in X_{n+1}$ be the unique trajectory with $(x,u)\in\mathcal{T}_{n+1}$. Then
		\[
		\left\{y\in X_{n+1}\colon (y,u)\in\mathcal{S}_{n+1}\right\}=\{x\},
		\]
		so clearly $u$ is essential if and only if $x$ is an isolated point and \eqref{EssentialCharacterization2} holds.
	\end{proof}
	
	The next two lemmas show that all trajectories may be coded with essential elements of $\A$ and that there is a certain uniqueness to the coding by essential elements.
	
	\begin{lemma} \label{Lemma:EssentialCovers}
		Let $E$ be the set of essential symbolic sequences in $\Sigma_M$. For all $x\in X$ there exists $\mathbf{a}\in E$ such that $(x_i,x_{i+1})\in G(a_i)$ for all $i\geq 0$.
	\end{lemma}
	
	\begin{proof}
		Let $x\in X$, and let $\mathbf{a} \in \Sigma_M$ such that $(x_i,x_{i+1})\in G_0(a_i)$ for all $i\geq 0$. If $\mathbf{a}\in E$, then we are done. Otherwise $\mathbf{a}\in\Sigma_M\setminus E$, so we may fix $n_0\geq 0$ to be the smallest integer such that $a_{n_0}\in\A\setminus\A^{(e)}$. By Lemma~\ref{Lemma:EssentialCharacterization}, Part~\eqref{EssentialCharacterization1}, $a_{n_0}\in\A_2$. For all $i\geq n_0$, the word $a_0\cdots a_i$ includes $a_{n_0}$ and is thus not essential, so we must have that $a_i\in\A_2$ for all $i\geq n_0$ as well. Now let $m_0\leq n_0$ be the smallest integer such that $a_i\in\A_2$ for all $i\geq m_0$.
		
		For each $k\geq n_0$, let $u^k=a_{m_0}\cdots a_k$, which is not an essential word (by choice of $n_0$). Also, since $a_i \in \A_2$ for $i \geq m_0$, we see that $G_0(a_i)$ is a singleton, and then $(x_{m_0},\ldots, x_{k+1})$ is the unique trajectory satisfying $(x_i,x_{i+1})\in G_0(a_i)$. Since $u^k$ is not essential, there is a word $v^k=v^k_{m_0}\cdots v^k_k$ such $v^k_k\in\A_0\cup\A_1$, and $(x_i,x_{i+1})$ is an endpoint of $G(v^k_i)$. Now we may extend $v^k$ to an infinite sequence $\mathbf{b}^k=(b^k_i)_{i=m_0}^\infty \in \Sigma_M$ such that $b^k_i=v^k_i$ for $m_0\leq i\leq k$. In this way we create a sequence $(\mathbf{b}^k)_{k=n_0}^\infty$ in the compact space $\Sigma_M$, and thus it must have an accumulation point $\mathbf{b}=(b_i)_{i=m_0}^\infty\in\Sigma_M$.
		
		We show that $b_i\in\A^{(e)}$ and $(x_i,x_{i+1})\in G(b_i)$ for all $i\geq m_0$. Fix $i\geq m_0$. Then $b_i$ is an accumulation point of $(b^k_i)_{k=n_0}^\infty$, so $b^k_i=b_i$ for infinitely many $k$. In particular, there exists $k>i$ with $b_i=b^k_i=v^k_i$. We know that $v^k$ is an essential word, so every term in $v^k$ is essential. Thus $b_i\in\A^{(e)}$. By construction, $(x_i,x_{i+1})\in G(v^k_i)=G(b_i)$.
		
		If $m_0=0$, then $\mathbf{b}$ is our desired sequence, and we are done. Otherwise, $a_{m_0-1}\in\A_0\subset\A^{(e)}$, and $x_{m_0}$ is an interior point of $R(a_{m_0-1})$ (since $(x_{m_0-1},x_{m_0}) \in G_0(a_{m_0-1})$). We also have that $x_{m_0}\in D(b_{m_0})$, and then by the Markov property we have that $D(b_{m_0})\subset R_0(a_{m_0-1})$. Thus we may define $b_{m_0-1}=a_{m_0-1}$, and we have $b_{m_0-1}b_{m_0}\in\mathcal{L}_2$. From here, we may define $b_i=a_i$ for all $0\leq i\leq m_0-1$ to form an element $\mathbf{b^*}=(b_i)_{i=0}^\infty\in\Sigma_M$ such that $(x_i,x_{i+1})\in G(b_i)$ for all $i\geq 0$. By the choice of $n_0$, we have that $a_i\in\A^{(e)}$ for all $0\leq i<n_0$, so in particular it follows that $b_i\in\A^{(e)}$ for all $0\leq i<m_0$. Thus $\mathbf{b^*}\in E$.
	\end{proof}
		
	\begin{lemma}\label{Lemma:essential implies unique labeling}
		Given $n\in\N$,	if $u\in\mathcal{L}_n(\A)$ is essential, and $(x,u)\in\mathcal{S}_{n+1}$, then for every $v\in\mathcal{L}_n(\A)\setminus\{u\}$, we have $(x,v)\notin\mathcal{T}_{n+1}$.
	\end{lemma}
	\begin{proof}
		Let $n\in\N$, $u\in\mathcal{L}_n(\A)$, and $(x,u)\in\mathcal{S}_{n+1}$.	We have assumed that $F$ is properly parametrized, so for every $v\in\mathcal{L}_n(\A)\setminus\{u\}$, the set $\left\{y\in X_{n+1}\colon (y,v)\in\mathcal{S}_{n+1}\right\}$ is disjoint from $\left\{y\in X_{n+1}\colon (y,u)\in\mathcal{S}_{n+1}\right\}$. Then since $u$ is essential, we  have that $\left\{y\in X_{n+1}\colon (y,u)\in\mathcal{S}_{n+1}\right\}$ is open, so
		\[
		\closure_{X_{n+1}}\left(\left\{y\in X_{n+1}\colon (y,v)\in\mathcal{S}_{n+1}\right\}\right)\cap\left\{y\in X_{n+1}\colon (y,u)\in\mathcal{S}_{n+1}\right\}=\emptyset.
		\]
		Finally, by Remark~\ref{Remark:Closure},
		\[
		\closure_{X_{n+1}}\left(\left\{y\in X_{n+1}\colon (y,v)\in\mathcal{S}_{n+1}\right\}\right)=\left\{y\in X_{n+1}\colon (y,v)\in\mathcal{T}_{n+1}\right\},
		\]
		so we must have that $(x,v)\notin\mathcal{T}_{n+1}$ for all $v\in\mathcal{L}_n(\A)\setminus\{u\}$.
	\end{proof}
	
	Let $\pi_n\colon X\to X_n$ denote the projection map from the forward trajectory space to the space of finite trajectories of length $n$, defined by $\pi_n(x) = (x_0,\dots,x_{n-1})$. 
	
	\begin{lemma}\label{Lemma:SpecialLabeledTrajectories}
		Let $n \geq 2$. Suppose $U_n \subset X_n$ is open and $\pi_n^{-1}(U_n) \subset X$ is nonempty. Then there exists $(z_0^{n-1},a_0^{n-2}) \in \mathcal{S}_n$ such that $a_0^{n-2} \subset \mathcal{L}_n(E)$ and $z_0^{n-1} \in U_n$.
	\end{lemma}
	\begin{proof}
		Since $U = \pi_n^{-1}(U_n)$ is nonempty, there exists $x \in U$. By Lemma~\ref{Lemma:EssentialCovers}, there exists $\mathbf{a} \in E$ such that $(x_i,x_{i+1})\in G(a_i)$ for all $i\geq 0$. In particular, $(x_0^{n-1},a_0^{n-2})\in\mathcal{T}_n$
		
		If in fact $(x_0^{n-1},a_0^{n-2})\in\mathcal{S}_n$, then we are done. If not, the set $\{y\in X_n\colon (y,a_0^{n-2})\in\mathcal{T}_n\}$ is an arc with $x_0^{n-1}$ as an endpoint, and the interior of that arc is the set $\{y\in X_n\colon (y,a_0^{n-2})\in \mathcal{S}_n\}$. Since $x_0^{n-1}\in U_n$, which is open in $X_n$, there exists $y\in U_n$ such that $(y,a_0^{n-2})\in\mathcal{S}_n$.
	\end{proof}
	
	We have one final lemma before we can move on toour main results. This guarantees that essential words can always be extended to essential words of arbitrary length.
	
	\begin{lemma}\label{Lemma:samelength}
		If $u\in\mathcal{L}_n$ is an essential word, then for all $m\geq n$, there exists an essential word $v\in\mathcal{L}_m$ containing $u$.
	\end{lemma}
	\begin{proof}
		Let $u_0\cdots u_{n-1}\in\mathcal{L}_n$. It suffices to show that we may choose $u_n\in\A$ such that $u_0\cdots u_n$ is essential. The result then follows from induction.
		
		Case 1: Suppose $u_{n-1}\in\mathcal{A}_0\cup\A_1$. Then we may choose $u_n\in\A_0$ such that $D_0(u_n)\subset R_0(u_{n-1})$. Since $u_n\in\A_0$, Lemma~\ref{Lemma:EssentialCharacterization} yields that $u_0\cdots u_n$ is essential.
		
		Case 2: Suppose $u_{n-1}\in\mathcal{A}_2$. Then there is one unique trajectory $x\in X_{n+1}$ such that $(x,u)\in\mathcal{S}_{n+1}$, and from the definition of essential, $\{x\}$ is open in $X_{n+1}$. If there exists $u_n\in\A_1$ with $D(u_n)=R(u_{n-1})$, then $u_0\cdots u_n$ is an essential word extending $u_0\cdots u_{n-1}$. Otherwise,
		let $h\colon X_{n+2}\to X_{n+1}$ be projection onto coordinates $0$ through $n$ and observe that $h^{-1}(\{x\})$ is open in $X_{n+2}$. For every $y\in h^{-1}(\{x\})$, there must exist $u^y_n\in\A_2$ such that $\{(y_n,y_{n+1})\}=G(u^y_n)$. Since $\A_2$ is finite, $h^{-1}(\{x\})$ must be a finite set, so for every $y\in h^{-1}(\{x\})$, $\{y\}$ is open in $X_{n+2}$. Thus, for  any $u_n\in\A_2$ with $D(u_n)=R(u_{n-1})$, $u_0\cdots u_n$ is essential.
	\end{proof}
	
	\section{Topological transitivity and mixing}\label{Sect:TransandMixing}
	
	In this section we state and prove our first main results. The following result relates topological transitivity and density of periodic points of the forward trajectory space to properties of the associated SFT. 
	
\begin{thm} \label{Thm:Transitivity} Let $F$ be a properly parametrized Markov multi-map with forward trajectory space $X$ and associated SFT $\Sigma_M$.
		\begin{enumerate} 
			\item If $X$ is topologically transitive, then $\Sigma_M$ satisfies the Irreducibility Condition (IC).
			\item If $\Sigma_M$ satisfies the Coding Condition (CC)  and the Irreducibility Condition (IC), then $X$ is topologically transitive and has dense periodic points.
		\end{enumerate}
	\end{thm}
	
	As a consequence of this result, we have that if the Coding Condition (CC) holds, then the Irreducibility Condition (IC) is necessary and sufficient for the forward trajectory space to be topologically transitive. Our next main result address the topological mixing of the forward trajectory space.
			
	\begin{thm} \label{Thm:Mixing} Let $F$ be a properly parametrized Markov multi-map with forward trajectory space $X$ and associated SFT $\Sigma_M$.
		\begin{enumerate} 
			\item If $X$ is topologically mixing, then $\Sigma_M$ satisfies the Mixing Condition (MC).
			\item If $\Sigma_M$ satisfies the Coding Condition (CC) and the Mixing Condition (MC), then $X$ is topologically mixing and has dense periodic points.
		\end{enumerate}
	\end{thm}
	
	In the presence of the Coding Condition (CC), Theorem \ref{Thm:Mixing} yields that the Mixing Condition (MC) is necessary and sufficient for the forward trajectory space to be topologically mixing.
	
	Before we prove these theorems, let us state some equivalent conditions for the topological transitivity and mixing of the forward trajectory space. Observe that for all $x\in X$, $n\geq0$, and $\eps>0$, the set
	\[
	U_n(x,\epsilon) = \{ x' \in X \colon  \forall k = 0,\dots,n-1, \, |x_k - x'_k| < \epsilon \}
	\]
	is open in $X$. In fact, sets of this form are a basis for the topology on $X$. As such, the following lemmas are straightforward and we omit their proofs. Our first lemma characterizes topological transitivity and density of periodic orbits.
	
	\begin{lemma}\label{Lemma:TransCharacterization}
		$X$ is topologically transitive if and only if $\forall x \in X_{m}$, $\forall y \in X_{n}$, and $\forall \epsilon >0$, there exists $N$ and $z \in X_{m + N + n}$ such that
		\begin{itemize}
			\item for all $k = 0, \dots, m-1$, we have $|x_k - z_k| < \epsilon$, and
			\item for all $k = 0,\dots,n-1$, we have $|y_k - z_{m+N+k}| < \epsilon$.
		\end{itemize}
		
		Similarly, the periodic trajectories are dense in $X$ if and only if for all $ x\in X_m$ and $ \eps>0$, there exists $N$ and $z\in X_{2m+N}$ such that for all $k=0,\ldots,m-1$, we have $z_{k+N}=z_k$ and $|x_k-z_k|<\eps$.
	\end{lemma}
	
	Our next lemma characterizing topological mixing.
	
	\begin{lemma}\label{Lemma:MixingCharacterization}
		$X$ is topologically mixing if and only if $\forall x \in X_{m}$, for all $ y \in X_{n}$ and $ \epsilon >0$, there exists $N$ such that for all $j\geq N$, there exists $z \in X_{m + j + n}$ such that
		\begin{itemize}
			\item for all $k = 0, \dots, m-1$, we have $|x_k - z_k| < \epsilon$, and
			\item for all $k = 0,\dots,n-1$, we have $|y_k - z_{m+j+k}| < \epsilon$.
		\end{itemize}
	\end{lemma}

	\subsection{Topologial Transitivity and Density of Periodic Points}
	
	In this section, we prove Theorem \ref{Thm:Transitivity}. We prove the two parts (1) and (2) separately.
	
	\vspace{2mm}
	
	\begin{Pfof1-1-1}
		Assume the hypotheses of the theorem, and suppose $X$ is topologically transitive. To show that $\Sigma_M$ satisfies the Irreducibility Condition, it suffices to show that any two elements of $\A^{(e)}$ are connected by a word in $\mathcal{L}$. Let $a,b\in\A^{(e)}$, and let $u,v\in\mathcal{L}_n$ be essential words containing $a$ and $b$ respectively. (We are able to assume that $u$ and $v$ are the same length by Lemma~\ref{Lemma:samelength}.)  Let 
		\begin{align*}
			U_a&=\left\{x\in X_{n+1}\colon (x,u)\in\mathcal{S}_{n+1}\right\}\text{ and }\\
			U_b&=\left\{x\in X_{n+1}\colon (x,v)\in\mathcal{S}_{n+1}\right\}.
		\end{align*}
		Then $U_a$ and $U_b$ are nonempty open subsets of $X_{n+1}$. It follows that
		\begin{align*}
			W_a&=\left(U_a\times\prod_{i=n+1}^\infty[0,1]\right)\cap X\text{ and }\\
			W_b&=\left(U_b\times\prod_{i=n+1}^\infty[0,1]\right)\cap X
		\end{align*}
		are nonempty open subsets of $X$. Since $X$ is topologically transitive, there exists $k\geq n$ and a point $x\in\sigma^k(W_a)\cap W_b$.
		
		Thus $(x_0,\ldots, x_n)\in U_a$, and $(x_k,\ldots, x_{n+k})\in U_b$, so there must be some word $w=w_0\cdots w_{n+k-1}$ with $(x_i,x_{i+1})\in G(w_i)$ for all $0\leq i\leq n+k-1$. Thus by Lemma~\ref{Lemma:essential implies unique labeling}, $w_i=u_i$ for all $0\leq i\leq n-1$, and $w_i=v_{i-k}$ for all $k\leq i\leq n+k-1$. Thus $w$ contains a word beginning with $a$ and ending with $b$.
		
		Since this is true for all $a,b\in\A^{(e)}$, there is an irreducible component $\C_0$ containing $\A^{(e)}$ and $\Sigma_M$ satisfies the Irreducibility Condition.
	\end{Pfof1-1-1}~

	\begin{Pfof1-1-2}
		Suppose $\Sigma_M$ satisfies the Coding Condition and the Irreducibility Condition. Let $\mathcal{C}_0$ be the irreducible component with $\mathcal{A}^{(e)} \subset \mathcal{C}_0$, and let $Z\subset\Sigma_M$ be the  given by the Coding Condition. It follows from Property~\ref{CC1} of Definition~\ref{Def:CC} that every element of $\A_0$ appears in a sequence in $Z$. Also, $\A_0 \subset \A^{(e)}$. Then by the fact that $Z$ is transitive on symbols, we have that $Z\subset\Sigma_M(\C_0)$.
		
		Let $n,m\geq 0$, $x \in X_{m+1}$, $y \in X_{n+1}$, and $\epsilon >0$. By Lemma~\ref{Lemma:SpecialLabeledTrajectories}, there exist $(x^1,u^1) \in \mathcal{S}_{m+1}(\mathcal{C}_0)$ and $(y^1,v^1) \in \mathcal{S}_{n+1}(\mathcal{C}_0)$ such that 
		\begin{equation*}
			|x_k - x^1_k| < \frac{\eps}{2}, \quad \forall k = 0,\dots, m,
		\end{equation*}
		and
		\begin{equation*}
			|y_k - y^1_k| <\frac{\eps}{2}, \quad \forall k = 0, \dots, n.
		\end{equation*} 
		Using the continuity of $g_{u^1}$ and $g_{v^1}$, choose $\delta>0$ such that 
		\begin{itemize}
			\item if $x'_{m} \in R(u^1_{m-1})$ and $|x^1_m - x'_m| < \delta$ then $|x^1_k - x'_k| < \epsilon/2$ for all $k = 0,\dots, m-1$, where $x'_k = g_{u^1_k}\circ\cdots g_{u^1_{m-1}}(x'_m)$, and 
			\item if $y'_{n} \in R(v^1_{n-1})$ and $|y^1_n - y'_n| < \delta$ then $|y^1_k - y'_k| < \epsilon/2$ for all $k = 0,\dots, n-1$, where $y'_k = g_{v^1_k}\circ\cdots g_{v^1_{n-1}}(y'_n)$. 
		\end{itemize}

		By the Coding Condition, there exists  $N\geq 1$ such that for every word $w \in \mathcal{L}_{N}(Z)$ we have $\ell(I_{w}) < \delta$. Fix $w\in\mathcal{L}_{N}(Z)$ such that $x^1_m\in I_{w}$. Since $\C_0$ is irreducible, there exists $n'\geq 0$ and  $v' \in \mathcal{L}_{n'}(\mathcal{C}_0)$ such that $w v' v^1 \in \mathcal{L}(\mathcal{C}_0)$.
		
		Let $J=m+N+n'+n$. We will construct a finite trajectory $z\in X_J$ as per Lemma~\ref{Lemma:TransCharacterization}. Choose any $z_J\in R(u^1_{n-1})$ with $|z_J-y^1_n|<\delta$. For every $k=1,\ldots,n$, define $z_{J-k}=g_{v^1_{n-k}}\circ\cdots\circ g_{u^1{n-1}}(z_J)$. From our choice of $\delta$, we have that $|z_{J-k}-y^1_{n-k}|<\eps/2$ for all $k=0,\ldots,n$. Next, for each $k=1,\ldots,n'$, let $z_{J-n-k}=g_{v'_{n'-k}}\circ\cdots\circ g_{v'_{n'-1}}(z_{J-n})$. Similarly, for each $k=1,\ldots,N$, let $z_{J-n-n'-k}=g_{w_{N-k}}\circ\cdots\circ g_{w_{N-1}}(z_{J-n-n'})$.
		
		Observe that $J-n-n'-N=m$, and $z_m,x^1_m\in I_{w}$, so $|z_m-x^1_m|<\delta$. We continue on to define, for all $k=1,\ldots,n$, $z_{m-k}=g_{u^1_{n-k}}\circ\cdots\circ g_{u^1_{m-1}}(x^1_m)$, and we conclude, again from our choice of $\delta$, that for all $k=0,\ldots,m$, $|z_k-x^1_k|<\eps/2$.
		
		Thus we have constructed $z\in X_{J}$ such that for all $k=0,\ldots,m$, $|z_k-x^1_k|<\eps/2$, and for all $k=0,\ldots,n$, $|z_{m+N+n'+k}-y^1_k|<\eps/2$. Then by the triangle inequality and our choice of $x^1$ and $y^1$, we get $|z_k-x_k|<\eps$ for all $k=0,\ldots,m$ and $|z_{m+N+n'+k}-y_k|<\eps$ for all $k=0,\ldots,n$. Thus by Lemma~\ref{Lemma:TransCharacterization}, $X$ is topologically transitive.

		To see that periodic trajectories are dense in $X$, consider the previous proof but with $x=y$, $x^1=y^1$, and $u^1=v^1$. Then we get that $g_{wv'v}$ is a continuous function whose domain is the compact interval $R(v^1_{n-1})$ and whose codomain is  $D(w_0)\subset R(u^1_{m-1})=R(v^1_{n-1})$. As a continuous function from a compact interval to itself must a fixed point, we may choose our original point $z_J$ to be a fixed point of this function. Then the resulting finite trajectory $z\in X_{m+N+n'+n}$ is periodic.
	\end{Pfof1-1-2}
	
	\subsection{Topological Mixing}
	Here we present the proofs regarding topological mixing, which are very similar to those concerning topological transitivity.
	
	\vspace{2mm}
	
	\begin{Pfof1-2-1}
		Suppose $X$ is topologically mixing. Let $a,b\in\A^{(e)}$, and let $u,v\in\mathcal{L}_n$ be essential words containing $a$ and $b$ respectively. Let $0\leq j_1,j_2\leq n-1$ be integers with $u_{j_1}=a$ and $v_{j_2}=b$. Define the sets $U_a$, $U_b$, $W_a$, and $W_b$ as in the proof of Theorem~\ref{Thm:Transitivity} Part (1).
		
		Since $X$ is topologically mixing, there exists $N\geq n$ such that for all $k\geq N$, there exists a point $x^k\in\sigma^k(W_a)\cap W_b$. Then there is a word $w^k=w^k_0\cdots w^k_{n+k-1}$ with $(x^k_i,x^k_{i+1})\in G(w^k_i)$ for all $0\leq i\leq n+k-1$. By Lemma~\ref{Lemma:essential implies unique labeling}, this implies that $w^k_i=u_i$ for all $0\leq i\leq n-1$ and $w^k_i=v_{i-k}$ for all $k\leq i\leq n+k-1$.
		
		Then in particular $w_{j_1}w_{j_1+1}\cdots w_{k+j_2}$ is a word of length $j_2-j_1+k$ beginning with $a$ and ending with $b$. It follows that there is a word from $a$ to $b$ of any length $j\geq N+j_2-j_1$. Since this holds for all $a,b\in\A^{(e)}$, there is a mixing component $\C_0$ containing $\A^{(e)}$.
	\end{Pfof1-2-1}~

	\begin{Pfof1-2-2}
		Suppose $\Sigma_M$ satisfies the Coding Condition and the Mixing Condition. Let $\mathcal{C}_0$ be the mixing component with $\mathcal{A}^{(e)} \subset \mathcal{C}_0$, and let $Z\subset\Sigma_M$ be the subshift guaranteed by the Coding Condition. As in the proof of Theorem \ref{Thm:Transitivity}, we note that $Z \subset \Sigma_M(\C_0)$. 
		
		Let $n,m\geq 0$, $x \in X_{m+1}$, $y \in X_{n+1}$, and $\epsilon >0$. By Lemma~\ref{Lemma:SpecialLabeledTrajectories}, there exist $(x^1,u^1) \in \mathcal{S}_{m+1}(\mathcal{C}_0)$ and $(y^1,v^1) \in \mathcal{S}_{n+1}(\mathcal{C}_0)$ such that 
		\begin{equation*}
			|x_k - x^1_k| < \frac{\eps}{2}, \quad \forall k = 0,\dots, m,
		\end{equation*}
		and
		\begin{equation*}
			|y_k - y^1_k| < \frac{\eps}{2}, \quad \forall k = 0, \dots, n.
		\end{equation*} 
		Using the continuity of $g_{u^1}$ and $g_{v^1}$, choose $\delta>0$ such that 
		\begin{itemize}
			\item if $x'_{m} \in R(u^1_{m-1})$ and $|x^1_m - x'_m| < \delta$ then $|x^1_k - x'_k| < \epsilon/2$ for all $k = 0,\dots, m-1$, where $x'_k = g_{u^1_k}\circ\cdots g_{u^1_{m-1}}(x'_m)$, and 
			\item if $y'_{n} \in R(v^1_{n-1})$ and $|y^1_n - y'_n| < \delta$ then $|y^1_k - y'_k| < \epsilon/2$ for all $k = 0,\dots, n-1$, where $y'_k = g_{v^1_k}\circ\cdots g_{v^1_{n-1}}(y'_n)$. 
		\end{itemize}

		By the Coding Condition, there exists  $N_1\geq 1$ such that for every word $w \in \mathcal{L}_{N_1}(Z)$ we have $\ell(I_{w}) < \delta$. Fix $w\in\mathcal{L}_{N_1}(Z)$ such that $x^1_m\in I_{w}$. Since $\C_0$ is mixing, there exists $N_2\geq 1$ such that for all $j\geq N_2$ there is a word  $v' \in \mathcal{L}_{j}(\mathcal{C}_0)$ such that $w v' v^1 \in \mathcal{L}(\mathcal{C}_0)$. Fix  $j\geq N_2$ and $v'\in\mathcal{L}_{j}(\C_0)$ with $wv'v^1\in\mathcal{L}(\C_0)$.
		
		By repeating the same arguments as in the proof of Theorem~\ref{Thm:Transitivity} Part~(2), we construct a finite trajectory $z\in X_{m+N+j+n}$ such that $|z_k-x_k|<\eps$ for all $k=0,\ldots,m$ and $|z_{m+N+j+k}-y_k|<\eps$ for all $k=0,\ldots,n$. Since we may construct such an orbit for all $j\geq N$, we conclude by Lemma~\ref{Lemma:MixingCharacterization} that $X$ is topologically mixing.
	\end{Pfof1-2-2}

	\section{Specification}\label{Sect:Spec}
	
	We also establish sufficient conditions for $X$ to have the specification property. Recall that specification implies topological mixing, and for SFTs these properties are equivalent. Thus, it is clear that if $X$ has specification, then $\Sigma_M$ satisfies the Mixing Condition. Our final main result gives a sufficient condition for $X$ to have specification. Before stating this result, let us define a notion of uniform equicontinuity for families of functions defined on subintervals of $[0,1]$. Suppose $\{g_{\omega} : \omega \in \Omega\}$ is a family of functions indexed by $\Omega$, where each $g_{\omega}$ is defined on a compact (possibly degenerate) interval $I_{\omega} \subset [0,1]$. We say that the family is uniformly equicontinuous if $\forall \epsilon >0$, there exists $\delta >0$ such that for all $\omega \in \Omega$, if $x,y \in I_{\omega}$ and $|x-y| < \delta$, then $|g_{\omega}(x) - g_{\omega}(y)|<\epsilon$.
	
	\begin{thm} \label{Thm:Specification}
		Let $F$ be a properly parametrized Markov multi-map with forward trajectory space $X$ and associated SFT $\Sigma_M$. If $\Sigma_M$ satisfies the Coding Condition (CC) and the Mixing Condition (MC), and if the family $\{g_u\colon u\in\mathcal{L}(\C_0)\}$ is uniformly equicontinuous, then $X$ has the specification property.
	\end{thm}
	
	For some context, a traditional Markov map that is piecewise linear satisfies the equicontinuity condition of this theorem.
	
	Before we prove this theorem, we give a lemma characterizing specification for $X$ along the same lines as Lemma~\ref{Lemma:TransCharacterization} and Lemma~\ref{Lemma:MixingCharacterization}. Since this result is more technical, we provide a proof.
	
	We have not had need to discuss a metric on $X$, but it is convenient for the proof of this lemma. Given $x,y\in X$, define 
	\[
	d(x,y)=\sup_{k\geq0}\frac{|x_k-y_k|}{k+1}.
	\]
	Then $d$ is a metric on $X$ and the topology generated from this metric is the same as the topology $X$ inherits as a subspace of $[0,1]^{\Z_{\geq0}}$.
	
	\begin{lemma}\label{Lemma:SpecCharacterization}
		Suppose that for all $\eps>0$, there exists $N\geq 1$ such that: for all $n_1,\ldots,n_r\geq 0$, $x^1\in X_{n_1+1},\ldots,x^r\in X_{n_r+1}$ and $j_1,\ldots,j_r\geq N$, there exists $z\in X_{J+1}$ where $J=\sum_{k=1}^r(n_k+j_k)$ such that $z_J=z_0$ and for all $k=1,\ldots,r$, if we set $t_k=\sum_{i=1}^k(n_i+j_i)$, then $|x^k_i-z_{t_k+i}|<\eps$ for all $i=0,\ldots,n_k$. Then $X$ has specification
	\end{lemma}
	
	\begin{proof}
		Let $\eps>0$, and choose $N_1\geq 1$ according to the supposition of this lemma. Let $N_2$ be a positive integer with $(N_1+1)^{-1}<\eps$, and define $M=N_1+N_2$. Let $x^1,x^2,\ldots,x^r\in X$, let 
		\[
		a_1\leq b_1<a_2\leq b_2<\cdots<a_r\leq b_r
		\]
		be positive integers with $a_{k+1}-b_k\geq M$ for all $k=1,\ldots,r$, and let $J\geq b_r+M$ be an integer. For each $k=1,\ldots,r$, define $n_k=b_k-a_k+N_2$. Additionally, for $k=1,\ldots,r-1$, define $j_k=a_{k+1}-b_k+N_2$, and define $j_r=P-b_r+N_2$. Then each $n_k\geq 0$ and $j_r\geq N_1$. Finally, for each $k=0,\ldots r$, let $y^k=(x^k_{a_k},\ldots,x^k_{a_k+n_r})$.
		
		Observe that $\sum_{k=1}^r(n_k+j_k)=J$ and $\sum_{i=1}^k(n_i+j_i)=a_k$. By the supposition of this lemma, there exists $z\in X_{J+1}$ such that $z_0=z_J$, and for all $k=0,\ldots,r$ and $i=0,\ldots,n_k$, we have $|y^k_i-z_{a_k+i}|<\eps$. Since $z_0=z_J$, we may extend $z$ to an infinite trajectory
		\[
		z^*=(z_0,\ldots,z_{J-1},z_0,\ldots,z_{J-1},\ldots).
		\]
		Clearly $z^*$ is periodic with period $J$. Moreover, for each $k=0,\ldots,r$ and $a_k\leq i\leq b_k$, we have
		\[
		d\left(\sigma_X^i(x^k),\sigma_X^i(z^*)\right)=\sup_{j\geq 0}\frac{\left|x^k_{j+i}-z^*_{j+i}\right|}{j+1}.
		\]
		
		From the definition of $N_2$ and the fact that all terms in the trajectory are bounded by 1, we  have that 
		\[
		\sup_{j> N_2}\frac{\left|x^k_{i+j}-z^*_{i+j}\right|}{j+1}<\eps.
		\]
		Additionally,
		\[
		\left|x^k_{i+j}-z^*_{i+j}\right|=\left|y^k_{i+j-a_k}-z^*_{i+j}\right|<\eps,
		\]
		so we have that $d(\sigma_X^i(x^k),\sigma_X^i(z^*))<\eps$, and $X$ has the specification property.

	\end{proof}
	
	\begin{Pfof1-3}
		Suppose $\Sigma_M$ satisfies the Mixing Condition and the Coding Condition. Let $\C_0$ be the mixing component with $\A^{(e)}\subset\C_0$, and let $Z\subset\Sigma_M$ be the subshift that witnesses the Coding Condition. As in the proof of Theorem \ref{Thm:Transitivity}, we have that $Z \subset \Sigma_M(\C_0)$. Suppose additionally that the family of inverse functions $\{g_w\colon w\in\mathcal{L}(C_0)\}$ is uniformly equicontinuous. Let $\eps>0$, and choose $\delta\in(0,\eps)$ to satisfy the $\eps/2$-challenge to the uniform equicontinuity of $\{g_w\colon w\in\mathcal{L}(\C_0)\}$. From the Coding Condition (CC), there exists $N_1\geq 1$ such that for every word $w\in\mathcal{L}_{N_1}(Z)$, we have $\ell(I_w)<\delta$. Since $\C_0$ is a finite set and is a mixing component, there exists $N_2\geq 1$ such that for all $a,b\in\C_0$ and all $j\geq N_2$, there exists a word $u\in\mathcal{L}_j(\C_0)$ with $aub\in\mathcal{L}_{j+2}(\C_0)$. 
		
		Let $n_1,\ldots,n_r\geq 0$, let $x^1\in X_{n^1+1},\ldots,x^r\in X_{n^r+1}$, and let $j_1,\ldots,j_r\geq N_2$.  By Lemma~\ref{Lemma:SpecialLabeledTrajectories}, there exist $(y^1,u^1)\in\mathcal{S}_{n^1+1},\ldots,(y^r,u^r)\in\mathcal{S}_{n^r+1}$ such that for all $k=1,\ldots,r$ and $i=0,\ldots,n_k-1$, we have $|y^k_i-x^k_i|<\eps/2$.
		
		Choose words $w^1,\ldots,w^r\in\mathcal{L}_{N_1}(Z)$ such that for each $k=1,\ldots,r$, $x^k_{n^k}\in I_{w^k}$. Then by the Mixing Condition (MC), for each $k=1,\ldots,r-1$, we may choose $v^k\in\mathcal{L}_{j_k}(\C_0)$ such that $w^kv^ku^{k+1}\in\mathcal{L}(\C_0)$. Additionally, we may choose $v^r\in\mathcal{L}_{j_k}(\C_0)$ such that $w^rv^ru^1\in\mathcal{L}(\C_0)$. Then the function $g_{u^1w^1v^1u^2\cdots w^rv^r}$ is continuous, with domain $R(v^r_{j_r-1})\supset D(u^1_0)$ and codomain  $D(u^1_0)$. Thus this function has a fixed point $z'\in D(u^1_0)$. 
		
		Let $J=rN_1+\sum_{k=1}^r(n_k+j_k)$, and note that this is the length of the word $u^1w^1v^1\cdots u^rw^rv^r$. Define $z_J=z'$. Then for each $k=1,\ldots,j_r$, define $z_{J-k}=g_{v^{r}_{j_r-k}}\circ\cdots\circ g_{v^{r}_{j_r-1}}(z_J)$. For $k=1,\ldots,N_1$, define $z_{J-j_r-k}=g_{w^r_{N_1-k}}\circ\cdots\circ g_{w^r_{N_1-1}}(z_{J-j_r})$. We continue on in this manner, using the inverse functions associated to each symbol in our word of length $J$ to construct a finite trajectory $z\in X_{J+1}$.
		
		Moreover, since $z_J=z'$, the fixed point, we have that $z_J=z_{0}$, so $z$ is a periodic trajectory. Next set $t_1=n_1$, and for $k=2,\ldots,r$, set $t_k=(k-1)N_1+n_1+\sum_{i=1}^{k-1}(j_i+n_{i+1})$. Then we have for all $k=1,\ldots,r$, $y^k_{n_k},z_{t_k}\in I_{w^k}$, so $|y^k_{n_k}-z_{t_k}|<\delta$. Then from equicontinuity, we get $|y^k_{n_k-i}-z_{t_k-i}|<\eps/2$ for all $i=0,\ldots,n_k$. Therefore, for all $k=1,\ldots,r$ and $i=0,\ldots,n_k$, we have $|x^k_{n_k-i}-z_{t_k-i}|<\eps$, so by Lemma~\ref{Lemma:SpecCharacterization}, $X$ has specification.
	\end{Pfof1-3}

	\section{Additional results}\label{Sect:Additional}
	
	We begin this section with some results that establish sufficient conditions for the Coding Condition (CC), as well as the uniform equicontinuity condition of Theorem~\ref{Thm:Specification}. We begin with following sufficient condition for uniform equicontinuity. We assume throughout this section that $F$ is a properly parametrized Markov multi-map with all associated objects defined as in Sections \ref{Sect:Intro} - \ref{Sect:MarkovDefinitions}.
	
	\begin{prop}\label{Prop:equicontinuitycondition}
		Suppose that for every $a\in\A_0$, the graph $G(a)$ is a straight line segment. Then the family of inverse functions $\{g_u\colon u\in\mathcal{L}\}$ is uniformly equicontinuous.
	\end{prop}
	\begin{proof}
		By assumption, for every $a\in\A_0$, we have that $f_a$ is a linear function with slope $\ell(R(a))/\ell(D(a))$, and therefore $g_a$ has slope $\ell(D(a))/\ell(R(a))$. More generally, if $u_0\cdots u_n\in\mathcal{L}$, then the slope of $g_u$ is
		\[
		\frac{\ell(D(u_n))}{\ell(R(u_n))}\frac{\ell(D(u_{n-1}))}{\ell(R(u_{n-1}))}\cdots\frac{\ell(D(u_0))}{\ell(R(u_0))}.
		\]
		
		Additionally, for each $0\leq i\leq n-1$, $D(u_{i+1})\subset R(u_i)$, so we get that the slope of $g_u$ is less than or equal to $\ell(D(u_0))/R(u_n)$. Thus if we let
		\[
		M=\max\left\{\frac{D(a)}{R(b)}\colon a,b\in\A_0\right\},
		\]
		then given $\eps>0$ we have that $\delta=\eps/M$ satisfies the $\eps$-challenge to uniform equicontinuity.
	\end{proof}

	We now consider sufficient conditions for the Coding Condition (CC) to be satisfied. Recall that for $a\in\A_0$, we have $R(a)=[p,q]$, where $p$ and $q$ are elements of the partition $P$ but are not necessarily adjacent in $P$. If there exists $p'\in P\cap R_0(a)$, then for any $b\in\A$ such that $ab\in\mathcal{L}$, we have that $I_{ab}$ is a proper subset of $I_a$. On the other hand, if $a\in\A_0$ and $R_0(a)\cap P=\emptyset$, then for any $b\in\A_0$ with $ab\in\mathcal{L}$, we have $I_{ab}=I_a$.
	
	\begin{prop}\label{Prop:codingcondition}
		Let $\mathcal{D}=\{a\in\A_0\colon R_0(a)\cap P\neq\emptyset\}$. Suppose $\mathcal{D}$ is non-empty and that
		\[
		\sup
		\left\{
		\frac{\ell\left(I_{u_0\cdots u_n}\right)}{\ell\left(I_{u_0\cdots u_{n-1}}\right)}\colon u\in \mathcal{L},u_{n-1}\in \mathcal{D},\ell(I_{u_0\cdots u_{n-1}})\neq 0
		\right\}<1.
		\]
		If there exists an irreducible component $\C_0\subset \A$ containing $\A_0$, then $\Sigma_M$ satisfies the Coding Condition (CC).
	\end{prop}
	
	\begin{proof}
		Since $\C_0$ is an irreducible component and finite, there exists $N\geq 1$ such that for all $a,b\in\C_0$, there exists $u\in\mathcal{L}_N(\C_0)$ beginning with $a$ and ending with $b$. Observe that if $a\in \A_0\setminus\mathcal{D}$, then $ab\in\mathcal{L}$ implies $b\in\mathcal{A}_0$. These facts allow us to define the set
		\[
		Z=\left\{\mathbf{a}\in\Sigma_M(\C_0)\colon\forall j\geq0,a_{j+i}\in \mathcal{D}\text{ for some }0\leq i\leq N\right\}.
		\]
		The irreducibility of $\C_0$ and the choice of $N$ imply that $Z$ is transitive on symbols.
		
		We now show that 
		\[
		\lim_{n\to\infty}\max_{u\in\mathcal{L}_n(Z)}\ell(I_u)=0.
		\]
		Given this goal, we need only focus on those words $u\in\mathcal{L}$ with $\ell(I_u)>0$. Let
		\[
		\gamma=\sup
		\left\{
		\frac{\ell\left(I_{u_0\cdots u_n}\right)}{\ell\left(I_{u_0\cdots u_{n-1}}\right)}\colon u\in \mathcal{L},u_{n-1}\in \mathcal{D},\ell(I_{u_0\cdots u_{n-1}})\neq 0
		\right\}.
		\]
		Observe that $\ell(I_u)\leq 1$ for all $u\in\mathcal{L}$. Thus if $n\geq N+1$ and $u\in\mathcal{L}_n$, then $\ell(I_u)\leq\gamma$. It follows that if $n\geq jN+1$, then $\ell(I_u)\leq\gamma^j$. We have assumed that $\gamma<1$, so this approaches 0 as $j\to\infty$.
		
		Next we must show that for all $y\in[0,1]$, there exists $\mathbf{a}\in Z$ and $x\in X$ such that $(x_i,x_{i+1})\in G(a_i)$ for all $i\geq 0$ and $x_0=y$. Let $y\in[0,1]$. From Definition~\ref{Defn:MarkovMultiMap}, there exists $a_0\in\A_0$ with $y\in D(a_0)$. We now build our desired sequence in $Z$ inductively. If $a_0\in\mathcal{D}$, then choose any $a_1\in\A_0$ with $y\in I_{a_0a_1}$. If $a_0\notin\mathcal{D}$, then choose a word $a_0\cdots a_{j}\in\mathcal{L}(\C_0)$ with $j\leq N-1$, $a_{j}\in\mathcal{D}$, and $a_i\in\C_0\setminus\mathcal{D}$ for all $0\leq i\leq j-1$. Now since $a_0\in\A_0\setminus\mathcal{D}$, $R_0(a_0)$ does not contain any partition elements, so $a_0$ cannot be followed by a symbol in either $\A_1$ or $\A_2$. Thus $a_1\in\A_0$. Likewise, if $a_1\in\A_0\setminus\mathcal{D}$, then $a_2\in\A_0$. Continuing on, we get that, in fact, $a_i\in\A_0\setminus\mathcal{D}$ for all $0\leq i\leq j-1$.	It follows that $I_{a_0\cdots a_{j}}=I_{a_0}$ which contains $y$. In either case we have $k\leq N-1$ and a word $a_0\cdots a_k$ with $y\in I_{a_0\cdots a_k}$ and either $a_k\in\mathcal{D}$ or $a_{k-1}\in\mathcal{D}$.
		
		Now suppose that for some $k\geq 1$, $a_0,\ldots,a_k\in\A_0$ have been defined so that $y\in I_{a_0\cdots a_k}$, either $a_k\in\mathcal{D}$ or $a_{k-1}\in\mathcal{D}$, and the word $a_0\cdots a_k$ contains no more than $N-1$ consecutive terms in $\A_0\setminus \mathcal{D}$. If $a_k\in\mathcal{D}$, choose any $a_{k+1}\in\A_0$ with $y\in I_{a_0\cdots a_ka_{k+1}}$. If $a_k\notin\mathcal{D}$, choose a word $a_k\cdots a_{k+j}\in\mathcal{L}$ with $j\leq N-1$, $a_{k+j}\in\mathcal{D}$, and $a_i\in\A_0\setminus\mathcal{D}$ for all $0\leq i\leq j-1$. It follows that
		$
		I_{a_0\cdots a_{k+j}}=I_{a_0\cdots a_k}
		$
		which contains $y$.
		
		In this way we define a sequence $\mathbf{a}\in Z$ such that $y\in I_{a_0\cdots a_j}$ for all $j\geq 0$. Since each $a_i\in\A_0$, the function $f_{a_i}$ is defined. Thus we may let $x_0=y$, and for every $i\geq 0$, $x_{i+1}=f_{a_i}(x_i)$. Then $x\in X$, $(x_i,x_{i+1})\in G(a_i)$ for all $i\geq 0$, and $x_0=y$. Therefore $\Sigma_M$ satisfies the Coding Condition (CC).
	\end{proof}
	
	We have the following  corollary to this proposition that mirrors Proposition~\ref{Prop:equicontinuitycondition}.
	
	\begin{cor}\label{Cor:codingcondition}
		Suppose that for every $a\in\A_0$, the graph $G(a)$ is a straight line segment. If there exists $a\in\A_0$ with $R_0(a)\cap P\neq\emptyset$, and there exists an irreducible component of $\A$ containing $\A_0$, then $\Sigma_M$ satisfies the Coding Condition (CC).
	\end{cor}
	\begin{proof}
		Once again let $\mathcal{D}=\{a\in\A_0\colon R_0(a)\cap P\neq\emptyset\}$. If $a\in\mathcal{D}$, then there exist $b_1,\ldots,b_k\in\A_0$ such that $ab_i\in\mathcal{L}$ for all $1\leq i\leq k$, and $I_{ab_1},\ldots,I_{ab_k}$ form a partition of $I_a$. Choose $i$ so that $\ell(I_{ab_i})/\ell(I_a)$ is maximized, and let $\gamma(a)$ be equal to this fraction. Observe that $\gamma(a)<1$.
		
		Because of the linearity assumption, given any $u_0\cdots u_n\in\mathcal{L}$ with $u_{n-1}\in\mathcal{D}$, we have $\ell(I_{u_0\cdots u_{n}})\leq \gamma(u_{n-1})\ell(I_{u_0\cdots u_{n-1}})$. Since $\mathcal{D}$ is finite, we may let $\gamma=\max_{a\in\mathcal{D}}\gamma(a)$. It follows that $\gamma<1$ and the conditions of Proposition~\ref{Prop:codingcondition} are satisfied.
	\end{proof}
	
	\subsection{Connections between the forward trajectory and inverse limit systems}
	
	In this section we establish connections between the inverse system and the forward trajectory system for Markov multi-maps, and as corollaries we obtain versions of main results in terms of inverse limit systems. Recall that we denote the forward trajectory system by $X$ with shift map $T$ and the inverse limit system by $Y$ with shift map $S$.
	
	\begin{prop}\label{ForwardTransitiveImpliesInverseTransitive}
		If $(X,T)$ is topologically transitive, then $(Y,S)$ is topologically transitive.
	\end{prop}
	\begin{proof}
		Let $x,y\in Y$ and $\eps>0$. Choose $n\in \N$ such that $\sup_{k\geq n+1} 1/(k+1)<\eps$. Consider the finite trajectories $(x_{-n+1},x_{-n+2},\ldots,x_{0})$ and $(y_{-n+1},\ldots,y_{0})$. These can be extended to forward trajectories $\tilde{x},\tilde{y}\in X$ where $\tilde{x}_k=x_{k-n+1}$ and $\tilde{y}_k=y_{k-n+1}$ for all $0\leq k\leq n-1$. Since $(X,T)$ is topologically transitive, there exists a trajectory $z\in X$ such that $d_X(\tilde{x},z)<\eps$ and $d(\tilde{y},T^j(z))<\eps$ for some $j\in\N$.
		
		In particular, it follows from the definition of the metric $d_X$ that for all $0\leq k\leq n-1$, $|\tilde{x}_k-z_k|<\eps(k+1)$ and $|\tilde{y}_k-z_{k+l}|<\eps(k+1)$. The transitivity of $(X,T)$ implies that $F$ is surjective. In particular, that implies we may the finite trajectory $(z_0,\ldots,z_{j+n-1})$ and extend it to an inverse trajectory $\tilde{z}\in Y$ where $\tilde{z}_k=z_{k+n-1}$ for all $-(n+j+1)\leq k\leq 0$. Then
		\begin{align*}
			d_Y(x,\tilde{z})&=\sup_{k\geq 0}\frac{|x_{-k}-\tilde{z}_{-k}|}{-k+1}\\
			&=\max\left\{\sup_{0\leq k\leq n}\frac{|x_{-k}-\tilde{z}_{-k}|}{k+1},\sup_{k\geq n+1}\frac{|x_{-k}-\tilde{z}_{-k}|}{k+1}\right\}\\
			&<\max\left\{\sup_{0\leq k\leq n}\frac{\eps(k+1)}{k+1},\sup_{k\geq n+1}\frac{1}{k+1}\right\}\\
			&<\eps.
		\end{align*}
		Likewise $d_Y(y,S^j(\tilde{z}))<\eps$, so $(Y,S)$ is topologically transitive.
	\end{proof}
	
	Every finite trajectory can be extended to an infinite forward trajectory. However surjectivity of the multi-map is necessary in order to extend every finite trajectory to an infinite inverse trajectory. We say a multi-map $F\colon[0,1]\to2^{[0,1]}$ is \emph{surjective} if for every $y\in[0,1]$ there exists $x\in[0,1]$ such that $y\in F(x)$. When we assume that the forward trajectory system is topologically transitive, that implies surjectivity of the multi-map. Transitivity of the inverse limit system does not imply that the multi-map is surjective however. This is because for a point $x\in[0,1]$ to be part of an inverse trajectory, it must have pre-images under all iterations of the multi-map. This is made more explicit in the following remark which also appears in \cite[Remark~2.6]{KellyMeddaugh2015}
	
	\begin{rmk}
		Let $F\colon[0,1]\to2^{[0,1]}$ be a multi-map, and let
		\[
		W=\bigcap_{n=1}^\infty F^n([0,1]).
		\]
		Then $W=F(W)$, and the inverse limit spaces $Y(F)$ and $Y(F|_W)$ are equal.
	\end{rmk}
	
	With this observation, one may establish the following result by making straightforward adaptations to the proof of Proposition~\ref{ForwardTransitiveImpliesInverseTransitive}.
	
	\begin{prop}
		Let $F\colon[0,1]\to2^{[0,1]}$ be a multi-map, and let
		\[
		W=\bigcap_{n=1}^\infty F^n([0,1]).
		\]
		If the inverse limit system is topologically transitive, then the forward trajectory system $X(F|_W)$ is topologically transitive.
	\end{prop}
	
	Likewise, we may use the same methods to conclude the following results. Note that a version of \eqref{spec-equivalence} appears in \cite[Theorems~4.1\&4.3]{RainesTennant2018}.
	
	\begin{thm}\label{Thm:forwardvsinverse}
		Let $F\colon[0,1]\to2^{[0,1]}$ be a multi-map, and let
		\[
		W=\bigcap_{n=1}^\infty F^n([0,1]).
		\]
		\begin{enumerate}
			\item $X(F|_W)$ is Devaney chaotic if and only if $Y(F)$ is Devaney chaotic.
			\item $X(F|_W)$ is topologically mixing if and only if $Y(F)$ is topologically mixing.
			\item\label{spec-equivalence} $X(F|_W)$ has the specification property if and only if $Y(F)$ has the specification property.
		\end{enumerate}
		Moreover, if $X(F)$ is Devaney chaotic, topologically mixing, or has specification, then $W=[0,1]$, and $X(F|_W)=X(F)$.
	\end{thm}
	
	The following corollary follows from Theorem~\ref{Thm:forwardvsinverse} combined with Theorems~\ref{Thm:Transitivity}--\ref{Thm:Specification}.
	
	\begin{cor}\label{Cor:MainThmsInverse}
		Let $F$ be a Markov multi-map with associated SFT $\Sigma_M$.
		\begin{enumerate}
			\item If $\Sigma_M$ satisfies the Irreducibility Condition (IC) and the Coding Condition (CC), then the inverse limit system $(Y,S)$ is Devaney chaotic.
			\item If $\Sigma_M$ satisfies the Mixing Condition (MC) and the Coding Condition (CC), then the inverse limit system $(Y,S)$ is topologically mixing.
			\item If $\Sigma_M$ satisfies the Mixing Condition (MC) (with mixing component $\C_0$) and the Coding Condition (CC), and if $\{g_a\colon a\in\C_0\}$ is uniformly equicontinuous, then the inverse limit system $(Y,S)$ has the specification property.
		\end{enumerate}
	\end{cor}
	
	\section{Examples}\label{Sect:Examples}
	
	In this section, we present two examples of Markov multi-maps. Both have SFTs satisfying the Mixing Condition (MC), but only the first has a forward trajectory system that is topologically mixing. The key difference is that the second example does not satisfy the Coding Condition (CC). This example illustrates that in the absence of the Coding Condition (CC), the properties of the associated SFT do not in general determine the properties of the forward trajectory system.
	
	\begin{center}
		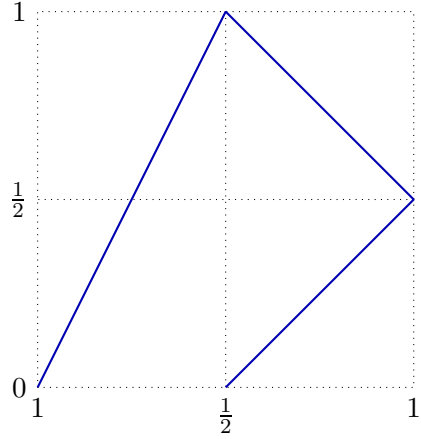
\begin{figure}
			\begin{tikzpicture}[scale=5]
				\draw[dotted] (0,0)node[left]{\small 0} -- (0,1)node[left]{\small 1} -- (1,1) -- (1,0)node[below]{\small 1} -- (0,0)node[below]{\small 1};
				\draw[dotted] (1/2,1) -- (1/2,0)node[below]{\small$\frac{1}{2}$};
				\draw[dotted] (1,1/2) -- (0,1/2)node[left]{\small$\frac{1}{2}$};
				\draw[blue!70!black,thick, join=bevel] (0,0) -- (1/2,1) -- (1,1/2) -- (1/2,0);
			\end{tikzpicture}
			\caption{Multi-map from Example~\ref{Example:EverythingWorks}}\label{Figure:EverythingWorks}
		\end{figure}
	\end{center}
	
	\begin{example}\label{Example:EverythingWorks}
		Let $F$ be the Markov multi-map pictured in Figure~\ref{Figure:EverythingWorks}. Then $P=\{0,1/2,1\}$, $\A_1=\emptyset$, and we may enumerate the elements of $\A_0$ and $\A_2$ by $\A_0=\{a_1,a_2,a_3\}$ and $\A_2=\{a_4,a_5,a_6,a_7\}$ where
		\begin{align*}
			D(a_1)&=\left[0,\frac{1}{2}\right] & R(a_1)&=[0,1]\\
			D(a_2)&=\left[\frac{1}{2},1\right] & R(a_2)&=\left[\frac{1}{2},1\right]\\
			D(a_3)&=\left[\frac{1}{2},1\right] & R(a_3)&=\left[0,\frac{1}{2}\right]\\
			D(a_4)&=\{0\} & R(a_4)&=\{0\}\\
			D(a_5)&=\left\{\frac{1}{2}\right\} & R(a_5)&=\{1\}\\
			D(a_6)&=\{1\} & R(a_6)&=\left\{\frac{1}{2}\right\}\\
			D(a_7)&=\left\{\frac{1}{2}\right\} & R(a_7) &=\{0\}
		\end{align*}
		The function $f_{a_1}\colon[0,1/2]\to[0,1]$ is given by $f_{a_1}(x)=2x$; $f_{a_2}\colon[1/2,1]\to[1/2,1]$ is given by $f_{a_2}(x)=3/2-x$; $f_{a_3}\colon[1/2,1]\to[0,1/2]$ is given by $f_{a+3}(x)=-3/2+x$.
		
		Then $\A^{(e)}=\A_0$ and $\mathcal{L}_2(\A_0)$ consists of the words $a_1a_1$, $a_1a_2$, $a_1a_3$, $a_2a_2$, $a_2a_3$, and $a_3a_1$. The set $\A_0$ is a mixing component. Moreover, since $G(a_1)$, $G(a_2)$, and $G(a_3)$ are all straight lines, and $R_0(a_1)\cap P\neq\emptyset$, we have by Corollary~\ref{Cor:codingcondition} that $\Sigma_M$ satisfies the Coding Condition. Additionally, by Propositions~\ref{Prop:equicontinuitycondition}, the family of inverse functions $\{g_u\colon u\in\mathcal{L}(\A_0)\}$ is equicontinuous. 
		
		Therefore by Theorem~\ref{Thm:Specification} and Corollary~\ref{Cor:MainThmsInverse}, both the forward and the inverse trajectory systems have the specification property.
	\end{example}

	\begin{example}\label{Example:NotTransitive}
		Let $F$ be the multi-map pictured in Figure~\ref{Figure:NotTransitive}. Once again $P=\{0,1/2,1\}$ and $\A_1=\emptyset$. Just as in Example~\ref{Example:EverythingWorks}, $\A^{(e)}=\A_0$, so we ignore $\A_2$. 
		We label the line segments clockwise beginning with the bottom left, so we have $\A_0=\{a_1,a_2,a_3,a_4\}$ where
		\begin{align*}
			D(a_1)&=\left[0,\frac{1}{2}\right] & R(a_1)&=\left[0,\frac{1}{2}\right]\\
			D(a_2)&=\left[0,\frac{1}{2}\right] & R(a_2)&=\left[\frac{1}{2},1\right]\\
			D(a_3) &=\left[\frac{1}{2},1\right] & R(a_3) &=\left[\frac{1}{2},1\right]\\
			D(a_4) &=\left[\frac{1}{2},1\right] & R(a_4) &=\left[0,\frac{1}{2}\right].
		\end{align*}
		We have $f_{a_1}(x)=1/2-x$, $f_{a_2}(x)=1/2+x$, $f_{a_3}(x)=3/2-x$, and $f_{a_4}(x)=-3/2+x$.
		
		Once again $\A_0$ is a mixing component of $\A$. However, in this case, $\Sigma_M$ does not satisfy the Coding Condition, and $X$ is not topologically transitive. To see that $\Sigma_M$ does not satisfy the Coding Condition, note that $I_u=I_{u_0}$ for every word $u\in\mathcal{L}(\A_0)$. This is because none of the functions have ranges which span multiple intervals from the partition. To see that $X$ is not topologically transitive, observe that every trajectory takes on at most four different values. Specifically, for all $x\in X$, there is some value $y\in[0,1/2]$ such that $x_i\in\{y,1/2-y,1/2+y,1-y\}$ for all $i\geq 0$.
	\end{example}
	
	\begin{center}
		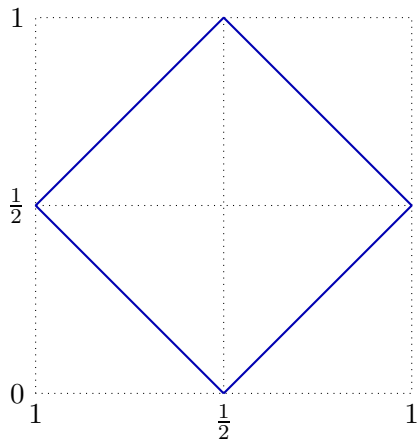
\begin{figure}
			\begin{tikzpicture}[scale=5]
				\draw[dotted] (0,0)node[left]{\small 0} -- (0,1)node[left]{\small 1} -- (1,1) -- (1,0)node[below]{\small 1} -- (0,0)node[below]{\small 1};
				\draw[dotted] (1/2,1) -- (1/2,0)node[below]{\small$\frac{1}{2}$};
				\draw[dotted] (1,1/2) -- (0,1/2)node[left]{\small$\frac{1}{2}$};
				\draw[blue!70!black,thick,join=bevel] (1/2,0) -- (0,1/2) -- (1/2,1) -- (1,1/2) -- (1/2,0);
			\end{tikzpicture}
			\caption{Multi-map from Example~\ref{Example:NotTransitive}}\label{Figure:NotTransitive}
		\end{figure}
	\end{center}

\end{document}